\documentclass[12pt]{amsart}

\usepackage{amsfonts,amssymb,enumerate,bm,hhline,makecell,array,mathtools}
\usepackage{mathrsfs}
\usepackage{comment}
\usepackage[normalem]{ulem}
\usepackage[colorlinks=true,citecolor=blue, urlcolor=blue, pagebackref, linkcolor=blue]{hyperref}
\usepackage{tikz-cd}
\usepackage{amscd}
\usepackage[shortlabels]{enumitem}
\setlist[itemize]{noitemsep,nolistsep}
\usepackage{tensor}
\usepackage{tikz}
\usetikzlibrary{matrix,arrows}
\usepackage{extarrows}
\usepackage[toc,page]{appendix}
\usepackage[all,ps,cmtip,rotate]{xy} 
\usepackage{color}

\def\cC{\mathscr{C}}

\def\Z{{\bf Z}}

\def\C{{\bf C}}

\def\R{{\bf R}}
\def\Q{{\bf Q}}

\def\phi{\varphi}

\DeclareMathOperator{\ch}{ch}

\DeclareMathOperator{\Sym}{Sym}

\DeclareMathOperator{\td}{td}
\DeclareMathOperator{\tp}{tp}

\def\llra{\hbox to 10mm{\rightarrowfill}}
\def\lllra{\hbox to 15mm{\rightarrowfill}}

\def\llla{\hbox to 10mm{\leftarrowfill}}
\def\lllla{\hbox to 15mm{\leftarrowfill}}

\def\hk{hyper-K\"ahler}
\def\hKm{hyper-K\"ahler manifold}

\newtheorem{lemm}{Lemma}[section]

\newtheorem{coro}[lemm]{Corollary}
\newtheorem{prop}[lemm]{Proposition}

\theoremstyle{remark}

\newtheorem{rema}[lemm]{Remark}
\newtheorem{conj}[lemm]{Conjecture}

\makeatletter
\@namedef{subjclassname@2020}{%
  \textup{2020} Mathematics Subject Classification}
\makeatother

\makeatletter
\def\@tocline#1#2#3#4#5#6#7{\relax
  \ifnum #1>\c@tocdepth 
  \else
    \par \addpenalty\@secpenalty\addvspace{#2}%
    \begingroup \hyphenpenalty\@M
    \@ifempty{#4}{%
      \@tempdima\csname r@tocindent\number#1\endcsname\relax
    }{%
      \@tempdima#4\relax
    }%
    \parindent\z@ \leftskip#3\relax \advance\leftskip\@tempdima\relax
    \rightskip\@pnumwidth plus4em \parfillskip-\@pnumwidth
    #5\leavevmode\hskip-\@tempdima
      \ifcase #1
       \or\or \hskip 1em \or \hskip 2em \else \hskip 3em \fi%
      #6\nobreak\relax
    \dotfill\hbox to\@pnumwidth{\@tocpagenum{#7}}\par
    \nobreak
    \endgroup
  \fi}
\makeatother

\def\lll{\ensuremath{\mathsf l}}
\def\mm{\ensuremath{\mathsf m}}

\makeatletter
\let\@wraptoccontribs\wraptoccontribs
\makeatother

\title{Numerical invariants of hyper-K\"ahler manifolds}

\begin{document}

\author[O.~Debarre]{Olivier Debarre}
\address{\parbox{0.9\textwidth}{Universit\'e   Paris Cit\'e, CNRS, Institut de Math\'ematiques de Jussieu-Paris rive gauche,\\[1pt]
8 Place Aur\'elie Nemours, 75013 Paris, France
\vspace{1mm}}}
\email{{olivier.debarre@imj-prg.fr}}

\address{\parbox{0.9\textwidth}{Shanghai Center for Mathematical Sciences \& School of Mathematical Sciences,\\[1pt] Fudan University, Shanghai 200438, China
\vspace{1mm}}}
\email{chenjiang@fudan.edu.cn}

\contrib[with an appendix by]{Chen Jiang}

\thanks{This project has received funding from the European
Research Council (ERC) under the European
Union's Horizon 2020 research and innovation
programme (Project HyperK --- grant agreement 854361).}

\begin{abstract}
We study  various constraints on the Beauville quadratic form and the  Huybrechts--Riemann--Roch  polynomial for \hKm s, mostly in dimension 6 and in the presence of an isotropic class. 

In an appendix, Chen Jiang proves that in general, the  Huybrechts--Riemann--Roch polynomial can  always be written as a linear combination with nonnegative coefficients of certain explicit polynomials with positive coefficients. This implies that the Huybrechts--Riemann--Roch polynomial satisfies a curious   symmetry property.
\end{abstract}

\maketitle

 \setcounter{tocdepth}{1}


\section{Introduction}

A {\em \hKm} is a simply connected compact K\"ahler manifold $X$  whose space of holomorphic 2-forms is spanned by a symplectic form.\ Its dimension  is necessarily an even number $2n$.\ A fundamental tool in the study of \hKm s is the {\em Beauville} form, 
 a canonical integral nondivisible nondegenerate quadratic form $q_X$ on the free abelian group $H^2(X,\Z)$ (\cite[th.~5]{bea}).\ Its signature is $(3,b_2(X)-3)$
 and  there is a positive rational number $c_X$ (the {\em Fujiki constant}) such that (\cite[Theorem~4.7]{fuj})
\begin{equation}\label{defqx}
\forall \alpha\in H^2(X,\Z)\qquad \int_X \alpha^{2n}=c_Xq_X(\alpha)^n.
\end{equation}

There exists a polynomial $ P_{RR,X}(T)$ (the {\em Huybrechts--Riemann--Roch polynomial})  with rational coefficients, leading term $\frac{c_X}{(2n)!}T^n$ and constant term $n+1$, such that, for every line bundle~$L$ on $X$,  one has (\cite[Corollary~3.18]{huy})
\begin{equation}\label{hrr}
\chi(X,L)=P_{RR,X}(q_X(c_1(L)))
.
\end{equation}
The  objects $q_X$, $c_X$, and $ P_{RR,X}(T)$ only depend on  the topology of $X$ and are in particular  deformation invariant.\
 
 In this note, we first prove in Section~\ref{sec:intro} a curious symmetry property for the polynomial $ P_{RR,X}(T)$ (Proposition~\ref{prop15}).\ This property also follows from a strengthening of \cite[Theorem~1.1]{jia} (which says that the polynomial $ P_{RR,X}(T)$ has positive coefficients)
 proved in the appendix by Chen Jiang.

  We then study a conjecture made in~\cite[Conjecture~1.4]{dhmv} (and proved in \cite[Theorem~1.5]{dhmv} when~$n=2$) about the possible values of $ P_{RR,X}(T)$ when the quadratic form~$q_X$ represents $0$.\ There exists then a nonzero class $\lll\in H^2(X,\Z)$   such that
$
\int_X\lll^{2n}=0$ and, for any $\mm\in H^2(X,\Z)$, if one writes $\int_X\lll^n\mm^n=a n!$, the number $a$ is necessarily an integer (\cite[Lemma~2.2]{dhmv}).\ The   conjecture deals with the case  $a=1$ (which happens for   $n$th punctual Hilbert schemes of  K3 surfaces and \hKm s of OG10 deformation type).

\begin{conj}[Debarre--Huybrechts--Macr\`i--Voisin]\label{conjnum}
Let $X$ be a \hk\ manifold of dimension $2n$ with  classes $\lll,\mm\in H^2(X,\Z)$ such that
\[
\int_X\lll^{2n}=0\quad\textnormal{and}\quad\int_X\lll^n\mm^n=n!.
\]
Then $c_X=(2n-1)!!$ and the Huybrechts--Riemann--Roch polynomial of $X$  is
\begin{equation}\label{valp}
 P_{RR,X}(T)= \binom{\frac12 T+1+n}{n}.
\end{equation}
\end{conj}


 Our main result is the following result (Proposition~\ref{prop43}) which almost  proves the conjecture (one would need to additionally prove that the case $n_X=2$ does not happen) in dimension~6 ($n=3$).

\begin{prop}
Let $X$ be a \hk\ manifold of dimension $6$ with  classes $\lll,\mm\in H^2(X,\Z)$ such that
$
\int_X\lll^{6}=0$ and $\int_X\lll^3\mm^3=3!
$.\ We have $q_X(\lll,\mm)=1$, 
  the quadratic form~$q_X$ is even, the Fujiki constant  $c_X$ is $15$, and
$$P_{RR,X}(T)=
  \binom{\frac{T}{2}+4}{3}-\frac{6-n_X}{16}T^2,$$
  where $n_X \in\{2,6 \}$.
  \end{prop}

 One may make the following more ambitious  conjecture for small positive values of $a$ (it is verified for all known examples of \hk\ manifolds and proved in general when  $n=2$ in~\cite[Theorem~9.3 and Theorem~1.5]{dhmv}).
 
 \begin{conj}\label{c2}
Let $X$ be a \hk\ manifold of dimension $2n$ with  classes $\lll,\mm\in H^2(X,\Z)$ such that
\[
\int_X\lll^{2n}=0\quad\textnormal{and}\quad\int_X\lll^n\mm^n=a n!\ , \quad\textnormal{with }\ a\in\{1,\dots,n\}.
\]
Then $a=1$ and $X$   is   of K3$^{[n]}$ or   OG10 deformation type.
\end{conj}


Again when $n=3$, we get in Proposition~\ref{prop45} a much weaker result in the case $a=2$ (which, according to Conjecture~\ref{c2}, should not occur at all).

\begin{prop} 
Let $X$ be a \hk\ manifold of dimension $6$ with  classes $\lll,\mm\in H^2(X,\Z)$ such that
$
\int_X\lll^{6}=0$ and $\int_X\lll^3\mm^3=2\cdot 3!
$.\ We have $q_X(\lll,\mm)=1$, 
  the quadratic form~$q_X$ is even,  the Fujiki constant  $c_X$ is $30$, and
  $$P_{RR,X}(T)=
   \frac{1}{24}T^3+\frac{n_X}{8}T^2+
 \bigl(\frac{4}{n_X}+\frac{n_X^2}{12}\bigr)T
  +4,$$
  where $n_X \in\{1,2,3,4 \}$.
\end{prop}

 \section{A symmetry property for the Huybrechts--Riemann--Roch polynomial }\label{sec:intro}

Let $X$ be a \hk\ manifold of dimension $2n$.\  In \cite[Definition~17]{nw} (see also  \cite[Definition~2.2]{jia}), Nieper-Wi\ss kirchen defined another quadratic form $\lambda_X$ on~$H^2(X,\R)$ (which is {\em not} integral on~$H^2(X,\Z)$).\ It satisfies (see~\cite[(5.18)]{nw})
\begin{equation}\label{fuj2}
\forall \alpha\in H^2(X,\Z)\qquad \frac{1}{(2n)!} \int_X \alpha^{2n}=A_X \lambda_X(\alpha)^n,
\end{equation}
where $A_X\coloneqq\int_X\td^{1/2}(X)$.\ 
By \cite[Proposition~10]{nw} and \cite[Proposition~2.3]{jia}, one can write
\begin{equation*}
q_X =m_X  \lambda_X,
\end{equation*}
where $m_X$ is a positive rational number, so that (compare~\eqref{defqx} and~\eqref{fuj2})
\begin{equation}\label{cxax}
c_X=\frac{(2n)!A_X}{m_X^n}. 
\end{equation}
We will also set $n_X\coloneqq 2m_X$.\ When $n>1$, one has (\cite[Section~6]{hs}, \cite[Corollary~5.5]{jia})
\begin{equation}\label{td121}
0< A_X <1.
\end{equation}

%


The Hirzebruch--Riemann--Roch theorem~\eqref{hrr}   takes the form 
\begin{equation}\label{hrrq}
\chi(X,L)=\int_X\td(X)\exp (c_1(L))=Q_{RR,X}(\lambda_X(c_1(L))) ,
\end{equation}
where $Q_{RR,X}(T)=P_{RR,X}(m_XT)$.\ 
%
The polynomial $Q_{RR,X}(T)$ was computed in \cite[Theorem~5.2]{nw} in terms of the Chern numbers of $X$.\  The formula is
\begin{equation}\label{nwiss}
Q_{RR,X}(T) = \int_X \exp\Bigl(-\sum_{k=1}^{+\infty}\frac{B_{2k}}{2k}  \ch_{2k}(X)T_{2k}\Bigl(\sqrt{\tfrac14T+1}\Bigr)\Bigr),
\end{equation}
where
\begin{itemize}
\item the $B_{2k}$ are the Bernoulli numbers;
\item the $\ch_{2k}\in H^{2k,2k}(X)$ are the homogeneous components of the Chern character of $X$;
\item the $T_{2k}(Y)$ are the (even) Chebyshev polynomials, defined by $T_{2k}( \cos \theta)=\cos 2k \theta$.
\end{itemize}

\bigskip
 
This formula implies curious symmetry relations for the  polynomials $P_{RR,X}(T)$ and $Q_{RR,X}(T)$ for which we have no geometric explanations.

\begin{prop}\label{prop15}
Let $X$ be a \hk\ manifold of dimension $2n$.\  The polynomial $Q_{RR,X}(T)$  satisfies the symmetry relation 
\begin{equation}\label{symm}
Q_{RR,X}(-T-4)=(-1)^n Q_{RR,X}(T).
\end{equation}
Equivalently, 
\begin{equation}\label{symmp}
P_{RR,X}(-T-2n_X)=(-1)^n P_{RR,X}(T).
\end{equation}
\end{prop}

When $n$ is odd, $-n_X$ is therefore a negative rational root of $P_{RR,X}(T)$.\ In all known examples, it is actually an integer (see 
also Lemma~\ref{le2}).

\begin{proof}
Let $P_k$ be the degree $k$ polynomial such that $P_k(T)=T_{2k}\Bigl(\sqrt{\tfrac14T+1} \Bigr)$.\  Set $\cos \theta\coloneqq \sqrt{\tfrac14T+1}$, so that $T=4(\cos^2\theta-1)=-4\sin^2\theta$.\ We compute 
\begin{multline*}
P_k(-T-4)=T_{2k}\Bigl(\sqrt{-\tfrac14T } \Bigr)=T_{2k}(\sin \theta) =T_{2k}( \cos( \theta-\tfrac{\pi}{2}))=\cos(2k( \theta-\tfrac{\pi}{2}))\\
{}=(-1)^k\cos 2k\theta=(-1)^kT_{2k}(\cos\theta)=(-1)^kT_{2k}\Bigl(\sqrt{\tfrac14T+1} \Bigr)=(-1)^kP_k(T).
\end{multline*}
By~\eqref{nwiss}, the polynomial $Q_{RR,X}(T)$ is a $\Q$-linear combination of polynomials of the type
\begin{equation*} 
P_{k_1}(T)  \cdots P_{k_r}(T)  \int_X    \ch_{2k_1}(X)\cdots  \ch_{2k_r}(X) 
\end{equation*}
for $k_1+\cdots +k_r=n$.\  The proposition therefore follows.
\end{proof}

 \begin{rema}\label{rema22}
 The symmetry relation~\eqref{symmp} implies that the polynomial $P_{RR,X}(T)$ is a linear combination with rational coefficients of the polynomials  $(T+n_X)^{n-2j}$, for $0\le j\le n/2$.\ Since its leading coefficient is $\frac{c_X}{(2n)!}$, we can write
 \begin{equation}\label{jc}
\begin{aligned} 
P_{RR,X}(T)&=\frac{c_X}{(2n)!} (T+n_X)^n +O(T^{n-2})\\
Q_{RR,X}(T)=P_{RR,X}(m_XT) &= A_X( T^n+2n T^{n-1} )+ O(T^{n-2}).
\end{aligned}
\end{equation}
The first two coefficients of $P_{RR,X}$ therefore determine $m_X$, $A_X$, and 
$c_X$ (see also \cite[Lemma~5.7]{jia}).

 Chen Jiang proves in Appendix~\ref{secapp}  that the polynomial $ Q_{RR,X}(T) $ is     a linear combination with {\em nonnegative} rational coefficients of the  polynomials
 $$Q_k(T)\coloneqq\sum_{j=0}^{k}\binom{k+j+1}{2j+1}T^j
 $$
 for $0\le k\le n$ and $n-k$ even.\ These polynomials satisfy the relation~\eqref{symm},\footnote{One has
 $Q_k\bigl(T+\tfrac1{T}-2\bigr) =\sum_{j=0}^kT^{2j-k}$.\ 
In particular, the polynomials $Q_k(T)$  satisfy~\eqref{symm} (change $T$ into~$-T$) and the roots of $Q_k(T)$ are the $k$ negative real numbers $-4\sin^2\frac{j\pi}{2(k+1)}$ for $1\le j\le k$, so that
$$Q_k(T)=\prod_{1\le j\le k}\Bigl(T+ 4\sin^2\frac{j\pi}{2(k+1)}\Bigr)
.$$
} so this much stronger  result implies Proposition~\ref{prop15}. 
  \end{rema}

\begin{coro}\label{coro23}
When $n=3$, one has
\begin{equation}\label{prr3} 
P_{RR,X}(T)= \frac{c_X}{720}(T+n_X)^3 +\Bigl(\frac{4}{n_X}-\frac{c_X}{720}n_X^2\Bigr) (T+n_X).
\end{equation}
\end{coro}

\begin{proof}
By Remark~\ref{rema22}, we can write
 $$P_{RR,X}(T)=\frac{c_X}{720}(T+n_X)^3+b(T+n_X), 
$$
where $b$  satisfies
$$\frac{c_X}{720}n_X^3+ bn_X=P_{RR,X}(0)
=4,$$
which gives the desired value for $b$.
\end{proof}

For all known examples of \hKm s $X$ of dimension $2n$, one has
\begin{equation*}
P_{RR,X}(T)= \binom{\frac12 T+1+n}{n}\qquad\textnormal{or}\qquad  
 (n+1)\binom{\frac12 T+n}{n}.
\end{equation*}
The roots of both of these polynomials are negative integers (this was conjectured to hold for all \hKm s in \cite[Conjecture~1.3]{jia}).\ In the next two remarks, we discuss what can be said about the reality of the  roots of the polynomial $P_{RR,X}(T)$ (or, equivalenty, of  $Q_{RR,X}(T)$) in dimensions 4 and 6 (when real, the roots are  negative, since both polynomials have positive coefficients).

 \begin{rema}[Real roots, $n=2$]
\label{en2}
When $n=2$, by~\eqref{jc}, we have
\begin{equation*}
Q_{RR,X}(T) = A_X(T^2+4 T)+3.
\end{equation*}
Easy computations (\cite[Lemma~4.1]{dhmv}) based on \cite[Main Theorem]{gua} give that
%
%
\begin{itemize}
\item either $b_2(X) =23$ and  $b_3(X) =0$, 
in which case 
 $A_X=\frac{25}{32}$, 
\item or $b_2(X) \le 8$, in which case 
$ \frac{5}{6}\le A_X\le \frac{131}{144}$.
\end{itemize}
In particular, the  discriminant 
$
4A_X(4A_X-3)$
of the polynomial $Q_{RR,X}(T)$ is positive, hence its roots are real. 
\end{rema}

  \begin{rema}[Real roots, $n=3$]
 When $n=3$, we have by Remark~\ref{rema22}
 $$Q_{RR,X}(T) =(T+2)(A_X(T^2+4T)+2).$$
 The roots of this polynomial  are all real if and only if   the discriminant
 $$16A_X^2-8A_X=8A_X(2A_X-1)$$
  is nonnegative, that is, if and only if $A_X\ge  \frac12$.\ The inequality $A_X>  \frac12$ is equivalent to the inequality (2) in \cite{beso}. It implies an upper   bound on $b_2(X)$.\ If $A_X\le  \frac12$, the class $c_2(X)$ is not in  the image of the morphism $\Sym^2\!H^2(X,\Q)\to H^4(X,\Q)$ (the Verbitsky component).

 \end{rema}

%
%
%
%

\section{Coefficients of the Huybrechts--Riemann--Roch polynomial}

For each positive integer $n$, we define the positive integer  
\[
C_n\coloneqq\gcd_{r_0,\dots,r_n\in \Z} \prod_{0\le j<k\le n}(r_j^2-r_k^2).
\]
One computes easily $C_1=1$, $C_2=12$, and, with a computer,\footnote{Many thanks to Jieao Song for making these computations.\ For any positive integer $n$, the   primes $p$ that divide~$C_n$ are exactly those such that $p\le  2n-1$ (this is because one can find $n+1$ distinct squares modulo $p$ if and only if $p>2n$).} 
\begin{align*}
C_3&=2^5\cdot 3^3\cdot 5,\\
C_4&= 2^{11} \cdot 3^5 \cdot 5^2 \cdot 7,\\
C_5&= 2^{18} \cdot 3^9 \cdot 5^4 \cdot 7^2,
\\
C_6&= 2^{27} \cdot 3^{14} \cdot 5^6 \cdot 7^3 \cdot 11,
\\
C_7&= 2^{37} \cdot 3^{19} \cdot 5^8 \cdot 7^5\cdot 11^2\cdot 13.
\end{align*}
Let $X$ be a \hk\ manifold of dimension $2n$.\  We write the Huybrechts--Riemann--Roch polynomial as
   \begin{equation*}
 P_{RR,X}(T) =:a_nT^n+\cdots +a_1T+a_0,
 \end{equation*}
 where $a_n= \frac{c_X}{(2n)!}$ and $a_0=n+1$.\ The proof of the following proposition uses the fact that the polynomial $ P_{RR,X}(T)$ takes integral values on every integer represented by $q_X$: this is because of the relation~\eqref{hrr} and the fact that, for every $\alpha\in H^2(X,\Z)$, there is a deformation of $X$ on which $\alpha$ becomes the first Chern class of a line bundle.

\begin{prop}\label{prop12}
Let $X$ be a \hk\ manifold of dimension $2n$.\  
For each $i\in\{0,\dots,n\}$, the coefficient $a_i$ of the polynomial $P_{RR,X}(T)$ belongs to $\frac{1}{2^iC_n}\Z$ (and to $\frac{1}{C_n}\Z$ if the quadratic form~$q_X$ is not even).\ 
In particular, the Fujiki constant $c_X$ is in $ \frac{(2n)!}{2^nC_n}\Z$.
\end{prop}

\begin{proof}
Let $q $ be an integer represented by $q_X$.\ 
For all $r_0,\dots,r_n\in \Z$, the integers $r_0^2q,\dots,r_n^2q$ are also represented by $q_X$, so that $P_{RR,X}(r_j^2q)=\sum_{i=0}^n  a_i r_j^{2i}q^i$ is an integer for all $j\in\{0,\dots,n\}$.\ 
The corresponding linear system with unknowns $a_0q^0,\dots,a_nq^n$ has a   Vandermonde matrix $(r_j^{2i})$, so we get
\[
a_iq^i\prod_{0\le j<k\le n}( r_j^2- r_k^2)   \in\Z
\]
for all $i\in\{0,\dots,n\}$, which implies $a_iq^iC_n\in\Z$.\ 
Since the  integral bilinear form associated with $q_X$ is  not  divisible, the gcd of all integers $q$ represented by $q_X$ is either $2$ (if the form $q_X$ is even) or $1$ (if it is not) and the proposition follows.
\end{proof}

In particular, we get $c_X\in\frac12\Z$ when $n=2$, and $c_X\in\frac1{48}\Z$ when $n=3$.\ For any $n$, Proposition~\ref{prop12} gives the lower bound $c_X\ge \frac{(2n)!}{2^nC_n}$, but what would be really interesting, in order to prove boundedness properties for \hk\ manifolds, would be to find  an upper bound on $c_X$ (see \cite{huyf}).

\begin{rema}\label{rem12}
Assume that $q_X$ represents all large enough even numbers (this is the case for all known examples).\ 
Then $P_{RR,X}(T)$ takes integral values on all large enough even numbers and this implies that its leading coefficient is in $\frac1{n!2^n}\Z$, hence $c_X\in (2n-1)!!\Z$.
\end{rema}

 \section{The Huybrechts--Riemann--Roch polynomial in the presence of an isotropic class}
 
 Let $X$ be a \hk\ manifold of dimension $2n$.\   Assume that there is an isotropic class $\lll\in H^2(X,\Z)$, that is, $q_X(\lll)=0$.\ For any   
  $ \mm\in H^2(X,\Z)$,
 $$a(\mm)\coloneqq\frac{1}{n!}\int_X\lll^n\mm^n$$
 is an integer (\cite[Lemma~2.2]{dhmv}) and 
\begin{equation} \label{cxa}
c_Xq_X(\lll,\mm)^n=a(\mm)\frac{(2n)!}{2^nn!}=a(\mm)(2n-1)!!.
\end{equation}
From now on, we assume $q_X(\lll,\mm)>0$.\ Using~\eqref{cxax} and~\eqref{td121}, we obtain
$$
m_X^n =\frac{(2n)!A_X}{c_X}<\frac{(2n)! }{c_X} = \frac{ 2^nn!q_X(\lll,\mm)^n}{a(\mm) }$$
hence
\begin{equation}\label{mxroot}
m_X < 2q_X(\lll,\mm) \sqrt[\leftroot{-3}\uproot{16}  {\scriptstyle {n}}]{\frac{  n!}{a(\mm) }}.
\end{equation}
 Using the bound $c_X\ge \frac{(2n)!}{2^nC_n}$, we also get $m_X< 2 \sqrt[\leftroot{-1}\uproot{2} n]{C_n}$.

\begin{lemm}\label{le1}
We have
$$
n! q_X(\lll,\mm)^n \mid a(\mm)C_n$$
and,  if $q_X$ is not even,
$$n! 2^n q_X(\lll,\mm)^n \mid a(\mm)C_n.$$

\end{lemm}


\begin{proof}
Using~\eqref{cxa}, we get
$$
a_n= \frac{c_X}{(2n)!} =\frac{a(\mm)}{2^nn!q_X(\lll,\mm)^n}.
$$
Then use Proposition~\ref{prop12}.\end{proof}

\begin{lemm}\label{le2}
 We have 
\begin{equation*}\label{eq}
a(\mm) \Bigl( \frac{q_X(\mm)+n_X}{2q_X(\lll,\mm)  }-\frac{n-1}2   \Bigr)     \in \Z.
\end{equation*}
In particular,
$$n_X\in   \Z+ \frac{2q_X(\lll,\mm) }{a(\mm)}\Z $$
so that $n_X $ is an integer when $ a(\mm)\in \{1,2\}$.
\end{lemm}

\begin{proof}
For every $t\in\Z$, the number
$$P(t )\coloneqq P_{RR,X}(q_X(t\lll+\mm))=P_{RR,X}(2tq_X(\lll,\mm)+q_X(\mm))
$$
is an integer.\ We have, using~\eqref{jc} and~\eqref{cxa},
   \begin{align*}\label{jc3}
P(t)&=\frac{c_X}{(2n)!} (2tq_X(\lll,\mm)+q_X(\mm)+n_X)^n +O(t^{n-2})\\
&= \frac{a(\mm) }{q_X(\lll,\mm)^n2^nn!} (2^n q_X(\lll,\mm)^nt^n+ n2^{n-1} q_X(\lll,\mm)^{n-1}(q_X(\mm)+n_X)t^{n-1}\bigr)  +O(t^{n-2})\\
&= \frac{a(\mm)}{  n!}    t^n+  \frac{a(\mm)}{q_X(\lll,\mm) 2 (n-1)!}    (q_X(\mm)+n_X)t^{n-1}  +O(t^{n-2}).
\end{align*}
This is an integer for all $t\in\Z$, hence so is
   \begin{align*} 
P(t)-a(\mm)\binom{t+n-1}{n}&=P(t)- a(\mm)\, \frac{ t^n+\frac{n(n-1)}2 t^{n-1}}{  n!}    +O(t^{n-2})\\
&= \Bigl( \frac{q_X(\mm)+n_X}{2q_X(\lll,\mm)  }-\frac{n-1}2   \Bigr)    \frac{a(\mm)}{ (n-1)!}  t^{n-1}   +O(t^{n-2}) .
\end{align*}
This implies the lemma.
\end{proof}

\subsection{Case  $a(\mm)=1$}  We know from~\cite[Theorem~1.5]{dhmv}
that in dimension 4, this case only  occurs when $X$ is of K3$^{[2]}$ deformation type.\ In particular, $P_{RR,X}(T)$ is then given by~\eqref{valp}.\ We believe (Conjecture~\ref{conjnum}) that the same should happen for any $n\ge2 $ (one would then have   $c_n=(2n-1)!!$ and $n_X=n+3$).\ We study the case $n=3$.

\begin{prop}\label{prop43}
Assume $n=3$ and $a(\mm)=1$.\ Then
$q_X(\lll,\mm)=1$, $c_X=15$, $n_X \in\{2,6 \}$, and the quadratic form $q_X$ is even.\ One also has 
$$P_{RR,X}(T)=
  \frac{1}{48}T^3+\frac{n_X}{16}T^2+
  \frac{13}6T+4
  = \binom{\frac{T}{2}+4}{3}-\frac{6-n_X}{16}T^2$$
and the sublattice $\Z\lll\oplus \Z\mm$ of $(H^2(X,\Z),q_X)$ is  a hyperbolic plane.
\end{prop}

\begin{proof}
%
 We have $C_3=2^5\cdot 3^3\cdot 5$ and we obtain   from Lemma~\ref{le1}
\begin{equation*}\label{eqdiv3}
q_X(\lll,\mm)^3\mid    2^4\cdot 3^2\cdot 5 \qquad\textnormal{(and $ q_X(\lll,\mm)^3 \mid 2 \cdot 3^2\cdot 5 $ if $q_X$ is not even),}
\end{equation*}
so that $ q_X(\lll,\mm)\in\{1,2\}$ (and $ q_X(\lll,\mm)=1 $ if $q_X$ is not even).
 
 \medskip
\noindent{\bf{Assume  $ q_X(\lll,\mm)=1 $.}}\ We have $c_X=15$ from~\eqref{cxa}, Lemma~\ref{le2} gives $q_X(\mm)+n_X\in 2\Z$, and~\eqref{mxroot} gives $m_X<2\sqrt[3]{6}\sim 3.6 $, so that $n_X \in\{1,2,3,4,5,6,7\}$.\ Furthermore, we have, by Corollary~\ref{coro23}, 
$$P_{RR,X}(T)=\frac{1}{48}(T+n_X)^3+\Bigl(\frac{4}{n_X}-\frac{1}{48}n_X^2\Bigr)(T+n_X).
$$
 For all values $q$ taken by $q_X$, this must be an integer when $T=q$, so that
\begin{equation}\label{eqdiv3b}
48n_X\mid   n_X(q+n_X)^3+ (192- n_X^3 )(q+n_X).
\end{equation}
 In particular, $n_X\mid 192 q$.\ If $n_X\in\{5,7\}$, this implies $n_X\mid  q$, which is impossible because   the gcd of all integers $q$ represented by $q_X$ is either $1$ or $2$.\ Otherwise, $16n_X\mid 192$, hence we obtain
\begin{equation}\label{eqdiv3bb}
16 \mid  (q+n_X)^3 - n_X^2  (q+n_X)=q(q+n_X)(q+2n_X).
\end{equation}
\begin{itemize}
\item When $n_X=1$, the relation~\eqref{eqdiv3bb} is equivalent to  $q\equiv 0,6,8,14,15\pmod{16}$.\ The case $q\equiv  15\pmod{16}$ is impossible since $4q$ is also represented but not in this list, hence $q\equiv 0,6,8,14\pmod{16}$ and $q_X$ is even.\ This contradicts the fact that $q_X(\mm)+n_X$ is even.
\item When $n_X=2$,  the relation~\eqref{eqdiv3bb} is equivalent to   $q $ even.
\item When $n_X=3$, the only possible odd value is $q\equiv 13\pmod{16}$.\ This implies that  $4q\equiv 4\pmod{16}$ should also be represented, but $4$ does not satisfy the relation~\eqref{eqdiv3bb}.\ So $q_X$ is even, which contradicts the fact that $q_X(\mm)+n_X$ is even.
\item When $n_X=4$,  the relation~\eqref{eqdiv3bb} is equivalent to  $4\mid q$, which is impossible because   the gcd of all integers $q$ represented by $q_X$ is either $1$ or $2$. 
\item When $n_X=6$,  the relation~\eqref{eqdiv3bb}  is equivalent to  $q $ even.
 \end{itemize}
All in all, we get $n_X \in\{2,6 \}$ and $q_X$ even.

\medskip

\noindent{\bf{Assume  $ q_X(\lll,\mm)=2$.}}\ The quadratic form $q_X$ is even, we have $c_X=\frac{15}8$ from~\eqref{cxa},   Lemma~\ref{le2} gives $\frac12 q_X(\mm)+ m_X\in 2\Z$, so that $m_X$ is an integer, and~\eqref{mxroot} gives $m_X<4\sqrt[3]{6}<8$, so that $m_X \in\{1,2,3,4,5,6,7\}$.\ As above, we  deduce from~\eqref{prr3} that
 $$ \frac{1}{8\cdot 48}(2q+ 2m_X)^3+\Bigl(\frac{2}{m_X}-\frac{1}{8\cdot 48}4m_X^2\Bigr)(2q+ 2m_X) $$
is an integer for all values $2q$ taken by $q_X$, so that 
$$48m_X\mid    m_X(q+m_X)^3+ (192- m_X^3 )(q+m_X).
$$
This is ``the same'' relation as~\eqref{eqdiv3b} and the discussion above allows us to conclude that $q$ must be even, so that all values taken by $q_X$ are divisible by $4$.\ This is   impossible because   the gcd of all values taken  by $q_X$ is  $2$.\ So this case does not occur. 
%
\end{proof}



\subsection{Case  $a(\mm)=2$}   We believe (Conjecture~\ref{c2})  this case should not     occur for any~$n\ge2$ and we know from~\cite[Theorem~9.3]{dhmv} that it does not when $n=2$.\ We study the case $n=3$.

\begin{prop}\label{prop45}
Assume $n=3$ and $a(\mm)=2$.\ Then,
 $q_X(\lll,\mm)=1$, $c_X=30$, $n_X \in\{1,2,3,4 \}$, and the quadratic form $q_X$ is even.\ One also has $P_{RR,X}(T)=
   \frac{1}{24}T^3+\frac{n_X}{8}T^2+
 \bigl(\frac{4}{n_X}+\frac{n_X^2}{12}\bigr)T
  +4$
and the sublattice $\Z\lll\oplus \Z\mm$ of $(H^2(X,\Z),q_X)$ is a hyperbolic plane.
\end{prop}

\begin{proof}
%
%
%
%
%
 We have $C_3=2^5\cdot 3^3\cdot 5$ and we obtain   from Lemma~\ref{le1}
\begin{equation*} 
q_X(\lll,\mm)^3\mid    2^5\cdot 3^2\cdot 5 \qquad\textnormal{(and $ q_X(\lll,\mm)^3 \mid 2^2 \cdot 3^2\cdot 5 $ if $q_X$ is not even),}
\end{equation*}
so that $ q_X(\lll,\mm)\in\{1,2\}$ (and $ q_X(\lll,\mm)=1 $ if $q_X$ is not even).

\medskip

\noindent{\bf{Assume  $ q_X(\lll,\mm)=1 $.}}\ We have $c_X=30$ from~\eqref{cxa}, Lemma~\ref{le2} gives $ n_X\in \Z$, and~\eqref{mxroot} gives $m_X<2\sqrt[3]{3}\sim 2.9 $, so that $n_X \in\{1,2,3,4,5\}$.\ Furthermore, we have, by~\eqref{prr3}, 
$$P_{RR,X}(T)=\frac{1}{24}(T+n_X)^3+\Bigl(\frac{4}{n_X}-\frac{1}{24}n_X^2\Bigr)(T+n_X),
$$
 For all values $q$ taken by $q_X$, this must be an integer when $T=q$, so that
\begin{equation*}
24n_X\mid   n_X(q+n_X)^3+ (96- n_X^3 )(q+n_X).
\end{equation*}
 In particular, $n_X\mid 96 q$.\ If $n_X=5$, this implies $n_X\mid  q$, which is impossible because   the gcd of all integers $q$ represented by $q_X$ is either $1$ or $2$.\ Otherwise, $8n_X\mid 96$, hence we obtain
\begin{equation}\label{eqdiv3cc}
8 \mid  (q+n_X)^3 - n_X^2  (q+n_X)=q(q+n_X)(q+2n_X).
\end{equation}
\begin{itemize}
\item When $n_X=1$, the relation~\eqref{eqdiv3cc} is equivalent to  $q\equiv 0,2,4,6,7\pmod{8}$; this means that every odd value taken by $q_X$ is $\equiv  7\pmod{8}$.\ Assume there exists $\alpha$ such that    $q_X(\alpha)\equiv  7\pmod{8}$.\ Since $q_X(k\lll+\mm)= 2k+q_X(\mm)$, the integer $q_X(\mm)$ must be even and we may even assume $q_X(\mm)=0$.\    For all   $t,u\in\Z$, the integer
$$q_X(t\lll+u\mm+\alpha)=2tu+2tq_X(\lll,\alpha)+2uq_X(\mm,\alpha)+q_X(\alpha)
$$
is odd, hence is $\equiv  7\pmod{8}$.\ This implies $tu+ tq_X(\lll,\alpha)+ uq_X(\mm,\alpha)\equiv 0\pmod4$.\ Taking $t=1$ and $u=0$, we obtain $q_X(\lll,\alpha) \equiv 0\pmod4$; taking $t=0$ and $u=1$, we obtain $q_X(\mm,\alpha) \equiv 0\pmod4$; taking $t=u=1$, we obtain a contradiction.\ Hence $q\equiv 0,2,4,6\pmod{8}$ and $q_X$ is even.\ 
\item When $n_X=2$, the relation~\eqref{eqdiv3cc} is equivalent to  $q\equiv 0,2,4,6\pmod{8}$ and $q_X$ is even.
\item When $n_X=3$, the relation~\eqref{eqdiv3cc} is equivalent to  $q\equiv 0,2,4,5, 6\pmod{8}$.\ If the case $q\equiv 5\pmod{8}$ occurs, the same reasoning as in the case $n_X=1$,  $q\equiv 7\pmod{8}$, gives   a contradiction, hence $q_X$ is even.
\item When $n_X=4$, the relation~\eqref{eqdiv3cc} is equivalent to  $q\equiv 0,2,4,6\pmod{8}$ and $q_X$ is even.
 \end{itemize}

\medskip

\noindent{\bf{Assume  $ q_X(\lll,\mm)=2 $.}}\ The quadratic form $q_X$ is even, we have $c_X=\frac{15}4$ from~\eqref{cxa}, Lemma~\ref{le2} gives $m_X\in   \Z$, and~\eqref{mxroot} gives $m_X<5.8 $, so that $m_X \in\{1,2,3,4,5\}$.\  As above, we  deduce from~\eqref{prr3} that
 $$ \frac{1}{8\cdot 24}(2q+ 2m_X)^3+\Bigl(\frac{2}{m_X}-\frac{1}{8\cdot 24}4m_X^2\Bigr)(2q+ 2m_X) $$
 must be an integer for all values $2q$ taken by $q_X$, so that 
$$24 m_X\mid    m_X(q+m_X)^3+ (96- m_X^3 )(q+m_X).
$$
We reason as above to conclude that the integer~$q$ must be even,  so that all values taken by the quadratic form~$q_X$ are divisible by $4$.\ This is   impossible because   the gcd of all values taken  by $q_X$ is  $2$.\ So this case does not occur. 
 \end{proof}

\appendix

\section{Positivity of the Huybrechts--Riemann--Roch polynomial}\label{secapp}

\smallskip
\begin{center}by \textsc{Chen Jiang}\end{center}
\medskip

\newcommand\sigmabar{{\overline{\sigma}}}
\newcommand{\rounddown}[1]{\lfloor{#1}\rfloor}
\newcommand{\roundup}[1]{\lceil{#1}\rceil}

  Throughout this appendix, $X$ is a \hKm\ of dimension $2n$ and we fix a symplectic form $\sigma\in H^0(X, \Omega^2_X)$.\   The degree $n$  Huybrechts--Riemann--Roch polynomial  $P_{RR,X}(T)$ was defined in the introduction, and the polynomial $Q_{RR,X}(T)=P_{RR,X}(m_XT)$  in Section~\ref{sec:intro}.\ These polynomials were  proved in \cite[Theorem~1.1]{jia} to have positive coefficients.\ The purpose of this appendix is to prove a refinement of this result.\ For every nonnegative integer $k$, we define a degree $k$ monic polynomial with positive coefficients by
 $$Q_k(T)\coloneqq \sum_{j=0}^{k}\binom{k+j+1}{2j+1}T^j
 =T^k+2kT^{k-1}+\cdots + k+1
 .$$
  Our result is the following.
  
  \begin{prop}\label{propa1}
Let $X$ be a \hKm\ of dimension $2n>2$.\ There are nonnegative rational numbers $b_0, b_1, \dots, b_{\lfloor n/2\rfloor}$ 
such that
\begin{equation}\label{eq Q=sum Q}
Q_{RR, X}(T)=\sum_{i=0}^{\lfloor n/2\rfloor}b_iQ_{n-2i}(T).
\end{equation}
Moreover, $b_0=\int_X\td^{1/2}(X)>0$ and $b_1>0$. 
\end{prop}

  For any $\alpha\in H^2(X,\R)$, we have
\[
Q_{RR,X}(\lambda_X(\alpha))=\int_X \td(X) \exp(\alpha),
\]
where $\lambda_X$ is the quadratic form  on~$H^2(X,\R)$ discussed in Section~\ref{sec:intro}.\ Indeed, by~\eqref{hrrq}, this equality holds when $\alpha$ is the first Chern class of a line bundle on $X$.\ It then holds for each $\alpha\in H^2(X,\Z)$ because there is a deformation of $X$ on which $\alpha$ becomes the first Chern class of a line bundle.\ Finally, it holds for every $\alpha\in H^2(X,\R)$ since both sides are polynomial functions of $\alpha$.

Moreover, one has (\cite[Definition~17]{nw}, \cite[Definition~2.2]{jia})
\[ 
\lambda_X(\alpha)\coloneqq\begin{cases}\frac{24n \int_X \exp(\alpha)}{ \int_X c_{2}(X) \exp(\alpha)} & \text{if well-defined;}\\ 0 & \text{otherwise.}\end{cases} 
\]
For simplicity, we set $\lambda_\sigma\coloneqq\lambda_X(\sigma+\sigmabar)$. We know that $\lambda_\sigma>0$ (see \cite[Lemma~2.4(2)]{jia}). 

In \cite[Definition~4.1]{jia}, for any $0\leq k\leq n/2$,  
we defined a class 
\[
\tp_{2k}\coloneqq\sum_{i=0}^k\frac{(n-2k+1)!\td^{1/2}_{2i}\wedge(\sigma\sigmabar)^{k-i}}{(-\lambda_\sigma)^{k-i}(k-i)!(n-k-i+1)!}\in H^{4k}(X,\R)\] 
which is of Hodge type $(2k,2k)$.\ 
One important fact is that, by \cite[Corollary~4.4]{jia}, 
$$\int_X\tp_{2k}^2(\sigma\sigmabar)^{n-2k}\geq 0.$$

\begin{lemm}\label{lem C independent}
The numbers
\[
C_k\coloneqq\frac{\int_X\tp_{2k}^2(\sigma\sigmabar)^{n-2k}}{\lambda_{\sigma}^{n-2k}}
\]
are deformation invariants of $X$.\ In particular, $C_k$ is independent of the choice of $\sigma$.
\end{lemm}

Here we remark that we cannot directly apply \cite[Corollary~23.17]{huy} as $\tp_{2k}$ might no longer be of type $(2k,2k)$ on deformations of $X$.

\begin{proof}
By  definition of $\tp_{2k}$, the number $C_k$ can be written as
\[C_k=\sum_{i=0}^k\sum_{j=0}^ka_{ij}\,\frac{\int_X\td^{1/2}_{2i}\td^{1/2}_{2j}(\sigma\sigmabar)^{n-i-j}}{\lambda_{\sigma}^{n-i-j}},\]
where the $a_{ij}$ are constants depending only on $n,k,i,j$.\ By \cite[Corollary~23.17]{huy} and \cite[Proposition~2.3]{jia},  \[\frac{\int_X\td^{1/2}_{2i}\td^{1/2}_{2j}(\sigma\sigmabar)^{n-i-j}}{\lambda_{\sigma}^{n-i-j}}=\frac{(n-i-j)!^2}{(2n-2i-2j)!}\,\frac{\int_X\td^{1/2}_{2i}\td^{1/2}_{2j}(\sigma+\sigmabar)^{2n-2i-2j}}{\lambda_{\sigma}^{n-i-j}}\]
 only depends  on $\td^{1/2}_{2i}\td^{1/2}_{2j}$, $c_2(X)$, and $c_X$, which implies that $C_k$ is a  deformation invariant of $X$. 
\end{proof}

\begin{proof}[Proof of Proposition~\ref{propa1}]
From \cite[Proof of Theorem~5.1]{jia}, for any $0\leq m\leq n$, we have \begin{align*}
 \int_X\td_{2m}(\sigma\sigmabar)^{n-m} 
= \sum_{i=0}^{\rounddown{m/2}} \frac{(n-m)!^2}{\lambda_\sigma^{m-2i} (n-2i)!^2}\binom{2n-2i-m+1}{m-2i}\int_X(\tp_{2i})^2(\sigma\sigmabar)^{n-2i}.
\end{align*}
In other words, 
\begin{align*}
 \int_X\td_{2m}(\sigma+\sigmabar)^{2n-2m} 
= \sum_{i=0}^{\rounddown{m/2}} \frac{(2n-2m)!}{\lambda_\sigma^{m-2i} (n-2i)!^2}\binom{2n-2i-m+1}{m-2i}\int_X(\tp_{2i})^2(\sigma\sigmabar)^{n-2i}.
\end{align*}

Thus we have the following equalities:
 \begin{align*}
\int_X \td(X) \exp(\sigma+\sigmabar)
 =&\sum_{m=0}^n\int_X  \frac{1}{(2n-2m)!}\td_{2m}(X) (\sigma+\sigmabar)^{2n-2m}\\
 =&\sum_{m=0}^n\sum_{i=0}^{\lfloor m/2\rfloor}\frac{1}{\lambda_\sigma^{m-2i}(n-2i)!^2}\binom{2n-2i-m+1}{m-2i}\int_X (\tp_{2i})^2(\sigma\sigmabar)^{n-2i}\\
 =&\sum_{m=0}^n\sum_{i=0}^{\lfloor m/2\rfloor}\frac{1}{(n-2i)!^2}\binom{2n-2i-m+1}{m-2i}C_i \lambda_\sigma^{n-m}\\
=&\sum_{i=0}^{\lfloor n/2\rfloor}\frac{C_i}{(n-2i)!^2} \sum_{m=2i}^n\binom{2n-2i-m+1}{m-2i}\lambda_\sigma^{n-m}\\
=&\sum_{i=0}^{\lfloor n/2\rfloor}\frac{C_i}{(n-2i)!^2} \sum_{m=0}^{n-2i}\binom{2n-4i-m+1}{m}\lambda_\sigma^{n-m-2i}\\
=&\sum_{i=0}^{\lfloor n/2\rfloor}\frac{C_i}{(n-2i)!^2} \,Q_{n-2i}(\lambda_\sigma).
\end{align*}
In other words,  
$$ Q_{RR, X}(\lambda_\sigma)=\sum_{i=0}^{\lfloor n/2\rfloor}\frac{C_i}{(n-2i)!^2} \,Q_{n-2i}(\lambda_\sigma).
$$
Here  $C_i\geq 0$ by \cite[Corollary~4.4]{jia}.\ 
By Lemma~\ref{lem C independent}, $C_i$ is independent of the choice of $\sigma$, so after replacing $\sigma$ by $t\sigma$ for any $t\in \C^\times$, we can get an equality of polynomials 
$$ Q_{RR, X}(T)=\sum_{i=0}^{\lfloor n/2\rfloor}\frac{C_i}{(n-2i)!^2}\, Q_{n-2i}(T),
$$
which gives  the desired equation \eqref{eq Q=sum Q}.

The last assertion is a consequence of \cite[Corollary~5.2]{jia}.
\end{proof}

\end{document}